\documentclass[12pt,a4paper]{article}

\usepackage{amssymb,amsmath}
\usepackage[mathscr]{eucal}  %to use mathematical symbol
\usepackage{amsfonts,amssymb,latexsym,amsmath,amsthm,epsfig,amscd}

\newtheorem{lemma}{Lemma}[section]
\newtheorem{cor}{Corollary}
\newtheorem{claim}{Claim}
\newtheorem{theorem}{Theorem}
\newtheorem{pro}{Proposition}
\newtheorem{question}{Question}
\newtheorem {defn}{Definition}
\newtheorem{remark}{Remark}

\newcommand{\RN}{\mathbb{R}^N}
\newcommand{\RL}{\mathbb{R}^L}
\newcommand{\RH}{\mathbb{R}^H}

\newcommand{\x}{\mathbf{x}}
\newcommand{\X}{\mathbf{X}}
\newcommand{\y}{\mathbf{y}}
\newcommand{\Y}{\mathbf{Y}}

\newcommand{\A}{\mathbf{A}}
\newcommand{\B}{\mathbf{B}}

\newcommand{\xnm}{\|\x\|_{\infty}}
\newcommand{\ynm}{\|\y\|_{\infty}}
\newcommand{\Anm}{\|{\cal{A}}(\X)\|_{\infty}}
\newcommand{\Bnm}{\|{\cal{B}}(\Y)\|_{\infty}}
\newcommand{\M}{\vec{M}_{\nu -1}}

\begin{document}
\title{Schmidt's Game, Badly Approximable Linear Forms and Fractals}
\author{Lior Fishman} 
\maketitle

\begin{abstract}
We prove that for every $M,N\in \mathbb{N}$, if $\tau$ is a Borel, finite,
absolutely friendly measure supported
on a compact subset ${\cal{K}}$ of $\mathbb{R}^{MN}$,
then ${\cal{K}}\cap\textbf{BA}(M,N)$
is a winning set in Schmidt's game sense played on ${\cal{K}}$,
where $\textbf{BA}(M,N)$ is the set of 
badly approximable $M\times N$ matrices.
As an immediate consequence we have the following application. 
If ${\cal{K}}$ is the attractor of an irreducible finite family 
of contracting similarity maps of $\mathbb{R}^{MN}$ satisfying the open set condition, 
(the Cantor's ternary set, Koch's curve and Sierpinski's gasket to name a few
known examples), then
\begin{center}
dim${\cal{K}}=$ dim${\cal{K}}\cap \textbf{BA}(M,N)$.
\end{center}

\end{abstract}

\setcounter{section}{-1}

\section{Introduction}

\noindent In his paper \textsl{Badly Approximable Systems of Linear Forms} \cite{SSS}, 
W. M. Schmidt proved that the set of badly approximable 
$M\times N$ matrices in $\mathbb{R}^{MN}$
is uncountable and in fact of full Hausdorff dimension, i.e., $MN$.
His proof is based on what is now referred to as Schmidt's game, 
first introduced by Schmidt in \cite{S}.
More precisely, he proved that this set is $\frac{1}{2}$-winning, 
from which
the conclusion regarding the Hausdorff dimension (and thus the cardinality of this set) is drawn.  
(See \cite{SSS} for a comprehensive review of partial results obtained prior to \cite{SSS}).
In recent years similar questions have been posed regarding the intersection 
of badly approximable numbers and vectors with certain subsets of $\RN $. 
For example, let ${\cal{K}}$ be any of the following sets: 
Cantor dust, Koch's curve, Sierpinski's gasket, 
or in general, an attractor of an irreducible finite family of contracting 
similarity maps of $\mathbb{R}^{MN}$ 
satisfying the open set condition. 
(This condition due to J. E. Hutchinson \cite{H} is discussed in section 4). 
Denoting by $\textbf{BA}(M,N)$ the set of badly approximable $M\times N$ matrices , 
(where the case $N=1$ corresponds to 
badly approximable vectors)
one may ask the following questions:
\begin{enumerate}
	\item Is ${\cal{K}}\cap \textbf{BA}(M,N)\neq \emptyset$?
	\item If ${\cal{K}}\cap \textbf{BA}(M,N)\neq \emptyset$, what is dim$K\cap \textbf{BA}(M,N)$?
\end{enumerate}
Answers to both of these questions for $\textbf{BA}(M,1)$ have been independently given in \cite{KW} and \cite{KTV} 
and later strengthened by the author in \cite{LF} using Schmidt's game, 
proving dim$K\cap \textbf{BA}(M,1)$=dim${\cal{K}}$ for the family of sets mentioned above, but for the general case
where both $M$ and $N$ are strictly larger then one, the answer was, as far as we are aware of, unknown.

This paper's main results, theorem \ref{main theorem} and corollary \ref{application} 
generalize theses results
to the set of badly approximable matrices, hence proving an analogue to Schmidt's result in $\mathbb{R}^{MN}$.

We emphasize that the major difference, and for all practical purposes the only difference, between our proof and
that of Schmidt, is in lemma \ref{hard theorem} in our paper corresponding to lemma 4 in \cite{SSS}.
It is precisely in the proof of this lemma that player White has to specify his strategy. In Schmidt's paper this is done by
player White successively picking specific points in his opponent's previous chosen balls as the centers for his balls.
Unfortunately, we cannot follow this strategy simply by the fact that in any given ball centered on the support of our
measure, we have no way of determining whether a specific point belongs to the support of the measure 
(apart of course from the center point). Thus we have to resort to measure theoretic reasoning postulating the existence 
of ``good points'', i.e., points which could serve player White's strategy as centers for his balls.
This is done by utilizing results regarding absolutely friendly measures from 
D. Kleinbock, E. Lindenstrauss and B. Weiss, \textsl{On fractal measures and Diophantine approximation}
\cite{KLW}.   
Not originally intended for being a friendly environment for Schmidt's game, 
it turns out that the support of these measures is indeed hospitable to this game.\\

\noindent Section 5 is dedicated to a short discussion regarding 
the winning dimension of a set. (See section 5 for a formal definition). 
We show that Schmidt's optimal winning dimension result, windim$(\textbf{BA}(M,N)\cap \mathbb{R}^{MN})=\frac{1}{2}$ 
cannot be reproduced when playing on the Cantor ternary set.\\

\noindent In the last section we raise a question regarding the measure of the intersection of 
$\textbf{BA}(M,N)$
and the compact support of an absolutely friendly measure. We construct an example demonstrating the
need for additional research on the necessary  conditions for this measure of intersection to be 0.  

\subsection*{Acknowledgments}
\noindent  I thank the Israel Science Foundation through grant 2004149 for their support, the Clay Mathematics Institute and the BSF grant 2000247. 
\vspace{5mm}

\noindent My deepest gratitude to Dmitry Kleinbock for carefully reading this paper
and offering many insightful and helpful remarks as well as for inviting me to present my results
at the ``Shrinking Target Workshop''  held by the Clay Mathematics Institute.

\vspace{5mm}

\noindent Finally, It gives me great pleasure to thank Barak Weiss. 
I would not be in the least exaggerating in saying that without his help and support this paper would have not been written.

\section{Basic definitions, notations and formulation of main theorem}
\subsection{Linear forms}

%where $\gamma_{ij}\in \mathbb{R}$,

%\noindent Let 
%begin{equation}
%\begin{matrix}
%L_1(\x)=\gamma_{11}x_1+\ldots +\gamma_{1N}x_N \\
%. \\
%. \\
%. \\
%L_M(\x)=\gamma_{M1}x_1+\ldots +\gamma_{MN}x_N.
%\end{matrix}
%\label{linear forms}
%\end{equation}
%be a system of $M$ linear forms with $\gamma_{ij}\in \mathbb{R}$ for $1\leq i\leq M$ and $1\leq j\leq N$.\\

\noindent If $t\in \mathbb{R}$, let $<t>$ denote the distance of $t$ from the nearest integer.\\ 

%\noindent If $\x\in\mathbb{R}^{N}$, let ${\cal{L}}(\x)=(<L_{1}(\x)>,\ldots ,<L_{M}(\x)>)$\\

\noindent For $U\in\mathbb{R}^{D}$, $U=(u_{1}\dots u_{D})$ we define

\begin{center}
$dist(U,\mathbb{Z}^{D})=\|<u_{1}>,\dots , < u_{D}>\|_{\infty}$,
\end{center}

\noindent where for $V=(v_{1},\dots ,v_{D})$,
 $\|V\|_{\infty}=max_{1\leq i\leq D}\{|v_{1}|,\ldots ,|v_{D}|\}$.\\

\noindent Let  $M,N\in\mathbb{N}$ and let $A$ be a real $M\times N$ matrix.  
%\begin{center}
%$A=(\gamma _{ij})_{1\leq i \leq M, 1\leq j \leq N}$. 
%\end{center}

\noindent We say that $A$ is badly approximable
if there exists a real constant 

\noindent $0<C=C(A)$ such that for every $0\neq \x\in \mathbb{Z}^{N}$ 
 we have

\begin{equation}
dist(A \x ,\mathbb{Z}^{M})>C\|\x\|_{\infty}^{-\frac{N}{M}}.
\end{equation}
\vspace{3mm}

\noindent Denote by ${\textbf{BA}}(M,N)$ be the set of all $M\times N$ badly approximable matrices. \\

\noindent For the rest of this paper, 
if $U,V\in\mathbb{R}^{D}$ then $\left|U\right|$ 
is the usual vector length, i.e., $(\Sigma _{i=1}^{D}u_i^{2})^{\frac{1}{2}}$ and
$U\cdot V$ is the standard inner product. \\

%such that the forms $L_1,...,L_M$ given in (\ref{linear forms}) are badly approximable.\\

\subsection{Schmidt's game}
\label{sg}

Let $(X,d)$ be a complete metric space and let ${\cal{S}}\subset X$ be 
a given set (a target set). 
Schmidt's game \cite{S} is played by two players White and Black, 
each equipped with parameters $\alpha $ and $\beta $ 
respectively, $0<\alpha ,\beta <1$. 
The game starts with player Black choosing $y_0\in X$ and $\rho>0$ 
hence specifying a closed ball $U(0)=B(y_0,\rho)$. 
Player White may now choose any point $x_0\in X$ provided that 
$W(0)=B(x_0,\alpha \rho)\subset U(0)$. 
Next, player Black chooses a point $y_1\in X$ such that
$U(1)=B(y_1,(\alpha \beta)\rho)\subset W(0)$. 
Continuing in the same manner we have a 
nested sequence of non-empty closed sets  
$U(0)\supset W(0)\supset U(1)\supset W(1)\supset\ldots\supset U(k)\supset W(k)\ldots$ 
with diameters tending to zero as $k\rightarrow\infty$. 
As the game is played on a complete metric space, 
the intersection of these balls 
is a point $z\in X$. Call player White the winner if $z\in {\cal{S}}$. 
Otherwise player Black is declared winner. 
A strategy consists of specifications for a player's choices 
of centers for his balls 
as a consequence of his opponent's previous moves. 
If for certain $\alpha $ and $\beta $ 
player White has a winning strategy, i.e., 
a strategy for winning the game regardless of how well player Black plays, 
we say that ${\cal{S}}$ is an 
$(\alpha , \beta)$-winning set. 
If ${\cal{S}}$ and $\alpha$ are such that ${\cal{S}}$ is an $(\alpha , \beta)$-winning set 
for all possible $\beta $'s, we say that ${\cal{S}}$ is an 
$\alpha $-winning set. 
Call a set winning if such an $\alpha $ exists.\\

\noindent  We shall be considering a slight variant of 
Schmidt's original definition of the game, say $\cal{K}$-Schmidt's game
(and we thank B. Weiss for drawing our attention to this point).

Specifically, the game is played on a compact subset 
${\cal{K}}$ of $\RN$ and in his first move, player Black
specifies a point $x_{0}$ in ${\cal{K}}$ and a positive number $\rho$.
These choices uniquely determine a standard Euclidean closed ball $B(\x,\rho)$ 
centered at $\x$ and of radius $\rho$.
From this point on, center
points picked by either players are in ${\cal{K}}$ and the radii of their balls of choice, $\rho(B)$,
(where the balls are considered as balls in $\RN$) are well defined.\\ 

\noindent We emphasize that one could avoid using the notation $\rho(B)$ by
referring to, for example, player Black balls' 
radii as $(\alpha \beta)^{k}\rho$ for some $k\in \mathbb{N}$. \\

\noindent Finally we assume that the first ball $U(0)$ specified by player Black satisfies 
$\rho(U(0))\leq \text{diam}{\cal{K}}$. We remark that no loss of generality occurs by this assumption,
as otherwise player's White's strategy is to play arbitrarily until the first integer $k$ is reached with
$\rho(U_{k})\leq \text{diam}{\cal{K}}$.

\subsection{Absolutely friendly measures}

\noindent For our next definitions we assume $N\in\mathbb{N}$ and
${\cal{P}}\subset\RN$ is an affine subspace. We 
denote by $d_{{\cal{P}}}(\x)$
the Euclidean distance from $\x\in\RN$ to ${\cal{P}}$.

\noindent Given $\epsilon>0$, let 

\begin{center}
${\cal{P}}^{(\epsilon)}=\{\x\in\RN :d_{{\cal{P}}}(\x)<\epsilon\}$.
\end{center}

\begin{defn} 
Let $\tau$ be a Borel, finite measure on $\mathbb{R}^{N}$.
We say that $\tau$ is \textbf{absolutely friendly} if the 
following conditions are satisfied:

There exist constants $\rho_0$, $C$, $D$ and $a$ 
such that for every $0<\rho\leq \rho_0$ and for every $\x \in supp(\tau)$:
\begin{itemize}
\item[(i)] for any  $0<\epsilon\leq \rho$, and any affine hyperplane $\cal P$,

$\tau(B(\x,r)\cap{\cal P}^{(\epsilon)})< C(\frac{\epsilon}{\rho})^{^a}\tau (B(\x ,\rho))$. 

\item[(ii)] $\tau(B(\x,\frac{1}{2} \rho))> D\tau(B(\x ,\rho))$.

\end{itemize}

\end{defn}

\begin{remark}
The second condition is usually referred to as the doubling or Federer property.
The term ``absolutely friendly'' was first coined in \cite{PV} where stronger
assumptions regarding the definition of friendly measures (see \cite{KLW})
were needed.

\end{remark}

\subsection{Main theorem}

\begin{theorem} 
\label{main theorem}
For every $M,N \in\mathbb{N}$, if $\tau$ is a Borel, finite, absolutely
friendly measure on $\mathbb{R}^{MN}$, supported on a compact subset ${\cal{K}}$
of $\mathbb{R}^{MN}$, then ${\cal{K}}\cap \bf{BA}(M,N)$ is a winning
set in Schmidt's game sense, played on ${\cal{K}}$.
\end{theorem}

\vspace{10mm}

%\newpage

\section{Specific notations}

\noindent Set 

\begin{center}
$H=M\cdot N$ \hspace{10mm}and \hspace{10mm} $L=M+N$.\\
\end{center}

\noindent For the rest of the paper we shall assume $N\geq M$.
This assumption will not imply any loss of generality since in fact
``built in'' the proof is the fact that if the set of badly approximable $M\times N$
matrices is winning in ${\cal{K}}$-Schmidt's game, so is the set of their transposes.
%(See \ref{46} and \ref{46a}).\\

\noindent We shall be playing Schmidt's game on ${\cal{K}}$ 
as defined in theorem \ref{main theorem}, 
where we identify points in $\RH$ with 
$M\times N$ real matrices.\\ 

\noindent For $k\in\mathbb{N}$ we denote the $k$th ball chosen by player White by 
$W(k)$ and respectively player Black's balls by $U(k)$.\\ 

\noindent Let $\rho=\rho(U(0))$ be the first radius chosen by player Black.\\

\noindent Given a Borel, finite, absolutely friendly measure $\tau$ supported on
a compact subset ${\cal{K}}\subset \mathbb{R}^{H}$, define 

\begin{equation}
\sigma =\sigma({\tau})=3\cdot\text{max}\{\|X\|\ : X\in {\cal{K}}\}. 
\end{equation}

\begin{remark} 
In Schmidt's paper, $\sigma $ was defined as the maximal norm of a point in $U(0)$,
and thus determined by
$U(0)$. 
In our settings, as $\rho(U(0))<diam({\cal{K}})$, 
(see discussion regarding this assumption in subsection \ref{sg}), 
$\sigma$ is determined by ${\cal{K}}=supp(\tau)$ 
and thus constants involving $\sigma$
in Schmidt's paper are viewed as constants involving $\tau$.
\end{remark}

\noindent We assign boldface lower case letters ($\x$, $\y$,...etc.) 
to denote points in $\mathbb{R}^N$ and $\mathbb{R}^M$ while boldface upper case letters 
($\X$, $\Y_i$, $\B_i$...etc.) denote points in $\mathbb{R}^L$. 
Finally, upper case letters ($A$,$X$,$Y$,...etc.) denote points in $\mathbb{R}^H$.\\

%\newpage

\section{Proof of theorem}

\noindent The proof will be presented in the following order. In the first subsection
we shall begin by stating lemma \ref{hard theorem}, lemma \ref{hard theorem 2}
and derive corollary \ref{corollary of main lemma}. 
We shall then proceed and ultimately prove our main theorem, theorem \ref{main theorem}.
Once this result is established we shall prove lemma \ref{hard theorem} in the following section. 
(Theorem \ref{hard theorem 2} could be proved in an identical way to theorem \ref{hard theorem}).

\noindent The rationale behind this way of presentation is the following. 
Lemma \ref{hard theorem} and lemma \ref{hard theorem 2} are, 
to quote Schmidt when referring to the analogous lemmas in \cite{SSS}, difficult. 
One of the main difficulties is the need for seemingly obscure notations and definitions.
Furthermore, in our case, we shall also need to utilize some 
deeper results concerning absolutely friendly measures.
We hope that by demonstrating the relatively effortless way one derives the main theorem
once these lemmas are proved will convince the reader of their necessity.\\\\   
 
\subsection{Proof of theorem \ref{main theorem} assuming lemma \ref{hard theorem}}

\subsubsection{Yet some more notation}

\noindent We begin with some more notations and definitions.

\noindent For any 
\begin{center}
$A=\left(\begin{matrix}
\gamma_{11} & . & . & . & \gamma_{1N}\\
. & . & . & .& . \\
. & . & . & .& . \\
. & . & . & .& . \\
\gamma_{M1} & . & . & . & \gamma_{MN}
\end{matrix}\right)$
\end{center}

\noindent let

\begin{center}
$\A_1=(\gamma_{11},...,\gamma_{1N},1,0,...,0),...,
\A_M=(\gamma_{M1},...,\gamma_{MN},0,0,...,1)$
\end{center}

\begin{center}
$\B_1=(\gamma_{11},...,\gamma_{M1},1,0,...,0),...,
\B_N=(\gamma_{1N},...,\gamma_{MN},0,0,...,1)$.
\end{center}

\vspace{5mm}

\noindent Let $\X,\Y\in\mathbb{Z}^{L}$ of the form  
\begin{equation}
\X=(x_1,...,x_N,...,x_L) : \x=(x_1,...,x_N)\neq (0,...,0),
\label{special X}
\end{equation}

\begin{equation}
\Y=(y_1,...,y_M,...,y_L) :\y = (y_1,...,y_M)\neq (0,...,0).
\label{special Y}
\end{equation}

\noindent Set
\begin{equation}
{\cal{A}}(\X)=(|\A_{1}\cdot\X|,\ldots ,|\A_{M}\cdot\X|)
\label{special A}
\end{equation}

\noindent and
 
\begin{equation}
{\cal{B}}(\Y)=(|\B_{1}\cdot\Y|,\ldots ,|\B_{N}\cdot\Y|).
\label{special B}
\end{equation}\\

\noindent We notice that a matrix $A$ associated with a system of linear forms lies in 
${\textbf{BA}}(M,N)$ if and only if there exists a constant $C$ such that 
for all $\X$ such as in (\ref{special X}) 

\begin{equation}
\|\x\|_{\infty}^{N}\cdot\|{\cal{A}}(\X)\|_{\infty}^{M}>C.
\label{BA def}
\end{equation}

\vspace{10mm}

\noindent For a fixed $N$ and $v$, where $1\leq v \leq N$
and for any ${\cal{Y}} =\{\Y_1,...,\Y_N\}$,
there are ${\binom{N}{v}}^{2}$ matrices of the form 
\begin{equation}
\left( \B_{i_{k}}\cdot\Y_{j_{l}}\right)
\label{matrix 1}
\end{equation} 

\noindent where $1\leq i_1<...<i_v\leq N$ and $1\leq j_1<...<j_v\leq N$.

\noindent Define 

\begin{center}
$\vec{M}_{v, {\cal{Y}}}(A)\in \mathbb{R}^{{\binom{N}{v}}^2}$
\end{center}

\noindent as the vector whose components are 
the absolute value of the determinants of (\ref{matrix 1}) arranged
in some order.\\

\vspace{5mm}

\noindent Similarly for a fixed $M$ and $\nu$, where $1\leq v\leq M$
and for any ${\cal{Y^{'}}}=\{\Y_1,...,\Y_M\}$,
there are ${\binom{M}{v}}^{2}$ matrices of the form 
\begin{equation}
\left( \A_{i_{k}}\cdot\Y_{j_{l}}\right)
\label{matrix 2}
\end{equation} 

\noindent where $1\leq i_1<...<i_v\leq M$ and $1\leq j_1<...<j_v\leq M$.

\noindent Define 

\begin{center}
$\vec{M}^{'}_{v,{\cal{Y^{'}}}}(A)\in \mathbb{R}^{{\binom{M}{v}}^2}$
\end{center}

\noindent as the vector whose components are 
the absolute value of the determinants of (\ref{matrix 2}) arranged
in some order.\\

\noindent As we shall consequently see, we shall only be assuming that 
the elements of the sets ${\cal{Y}}$ and ${\cal{Y}^{'}}$ are orthonormal, but the proofs
DO NOT depend on a specific ${\cal{Y}}$ or ${\cal{Y^{'}}}$. (This is perhaps the most
important part of our main theorem's proof). 
\noindent Thus from this point on, for any fixed 
${\cal{Y}}$, we shall write
$\vec{M}_v(A)$ to mean $\vec{M}_{v, {\cal{Y}}}(A)$.\\

\noindent Define $\vec{M}_0(A)$, (similarly $\vec{M}^{'}_0(A)$) and 
$\vec{M}_{-1}(A)$, (similarly $\vec{M}^{'}_{-1}(A)$) as the 
one dimensional vector $(1)$.

\noindent For a closed ball $B\subset \mathbb{R}^{H}$, let

\begin{center}
$M_v(B)=max_{{A\in B}}\left|\vec{M}_v(A)\right|$ \hspace{5mm} and \hspace{5mm} 
$M_v^{'}(B)=max_{{A\in B}}\left|\vec{M}^{'}_v(A)\right|$.
\end{center}

\vspace{10mm}

\subsubsection{More on absolutely friendly measures}

\noindent For a ball $B\subset\RN$ and a real valued function $f$ on $\RN$, let

\begin{center}
$\|f\|_{B}=\text{sup}_{\x\in B}|f(\x)|$.\\
\end{center}

\noindent As an immediate consequence of 
proposition 7.3 in \cite{KLW} one has
the following corollary.

\begin{cor}
Let $\tau$ be a Borel, finite, absolutely friendly measure on $\mathbb{R}^{H}$. 
Then for every $k$ there exist $K=K(k)$ and $\delta=\delta(k)$ such that if 
$f$ is a real polynomial function on $\mathbb{R}^{H}$ 
of a bounded total degree 
$k$, then for any ball $B\subset\mathbb{R}^{H}$ centered on 
$supp(\tau)$ and any $\epsilon>0$, 
\begin{equation}
\label{good}
\tau(\{x\in B:|f(x)|<\epsilon\})\leq K(\frac{\epsilon}{\|f\|_{B}})^{\delta}\tau(B).
\end{equation}
\label{cor barak}
\end{cor}

\noindent Thus given a Borel, finite, absolutely friendly measure $\tau$ 
on $\mathbb{R}^{H}$ and a polynomial function of
total bounded degree $L$
with associated constants $K=K(L)$ and $\delta =\delta(L)$ as 
in corollary \ref{cor barak} , let 

\begin{center}
$0<\epsilon_{0}=\epsilon_{0}(\tau,M,N)$  
\end{center}

\noindent be small enough as to satisfy

\begin{equation}
K(\epsilon_{0})^{\delta}<\frac{1}{2}.
\label{epsilon_0}
\end{equation}

\subsubsection{Statement of main lemmas}

\begin{lemma}
Given $\tau$, a Borel, finite absolutely friendly measure 
with $supp(\tau)={\cal{K}}$, where ${\cal{K}}$
is a compact subset of $\mathbb{R}^{H}$, 
we play Schmidt's game on ${\cal{K}}$ such that all balls 
chosen by the two players are centered on ${\cal{K}}$.    
Let $\epsilon_{0}$ be  
as defined in (\ref{epsilon_0}).
Then for any $\psi>0$, there exists
\begin{center}
$0<\alpha_{1} =\alpha_{1} (M,N,\psi,\tau)$, 
\end{center}

\noindent and for any $0<\beta<1$, $0\leq \nu \leq N$ , there exists

\begin{center}
$\mu_{\nu}=\mu_{\nu}(M,N,\alpha_{1}, \beta,\psi,\tau)$
\end{center}

\noindent such that for any $\Y_{1},...,\Y_{N}$ orthonormal vectors in $\mathbb{R}^{L}$,
if a ball $U\subset\mathbb{R}^{H}$ satisfying 
$\rho(U)=\rho_{0}<1$ is reached by player Black at some stage of the $(\alpha_{1},\beta)$ 
game, then 
player White has a strategy enforcing the first of player Black's ball $U(i_{\nu})$ with
\begin{center}
$\rho(U(i_{\nu}))<\rho_{0}\mu_{\nu}$
\end{center}
to satisfy for every $A\in U(i_{\nu})$
\begin{equation}
|\vec{M}_{\nu}(A)|>\left(\frac{\epsilon_{0}}{2}\right)^{\nu}\psi\rho_{0}\mu_{\nu} M_{\nu -1}(U(i_{\nu})).
\label{46}
\end{equation}
\label{hard theorem}
\end{lemma}

\noindent The following lemma can be proved almost exactly as 
lemma \ref{hard theorem}, substituting $M$ for $N$ in the appropriate places. 

\begin{lemma}
Given $\tau$, a Borel, finite absolutely friendly measure 
with $supp(\tau)={\cal{K}}$, where ${\cal{K}}$
is a compact subset of $\mathbb{R}^{H}$,
we play Schmidt's game on ${\cal{K}}$ such that all balls 
chosen by the two players are centered on ${\cal{K}}$.    
Let $\epsilon_{0}$ be  
as defined in (\ref{epsilon_0}).
Then for any $\psi>0$, there exists
\begin{center}
$0<\alpha_{2} =\alpha_{2} (M,N,\psi,\tau)$, 
\end{center}

\noindent and for any $0<\beta<1$, $0\leq \nu \leq M$ , there exists

\begin{center}
$\mu_{\nu}=\mu_{\nu}(M,N,\alpha_{2}, \beta,\psi,\tau)$
\end{center}

\noindent such that for any $\Y_{1},...,\Y_{M}$ orthonormal vectors in $\mathbb{R}^{L}$,
if a ball $U\subset\mathbb{R}^{H}$ satisfying 
$\rho(U)=\rho_{0}<1$ is reached by player Black at some 
stage of the $(\alpha_{2},\beta)$ game, then 
player White has a strategy enforcing the first of player Black's ball $U(i_{\nu})$ with
\begin{center}
$\rho(U(i_{\nu}))<\rho_{0}\mu_{\nu}$
\end{center}
to satisfy for every $A\in U(i_{\nu})$
\begin{equation}
|\vec{M}_{\nu}(A)|>\left(\frac{\epsilon_{0}}{2}\right)^{\nu}\psi\rho_{0}\mu_{\nu} M_{\nu -1}U(i_{\nu}).
\end{equation}
\label{hard theorem 2}
\end{lemma}

\subsubsection{Immediate corollary}

\begin{cor}
Given $\tau$, a Borel, finite absolutely friendly measure 
with $supp(\tau)={\cal{K}}$, where ${\cal{K}}$
is a compact subset of $\mathbb{R}^{H}$,   
we play Schmidt's game on ${\cal{K}}$ such that all balls 
chosen by the two players are centered on ${\cal{K}}$.  

\noindent There exists
\begin{center}
$0<\alpha =\alpha(M,N,\tau)$, 
\end{center}

\noindent and given any $0<\beta<1$, there exists

\begin{center}
$\mu=\mu(M,N,\alpha, \beta,\tau)$
\end{center}

\noindent such that for any $0<\mu ^{'}\leq \mu$ and for any $\Y_{1},...,\Y_{N}$ 
orthonormal vectors in $\mathbb{R}^{L}$,
if a ball $U\subset\mathbb{R}^{H}$ centered on ${\cal{K}}$ satisfying 
$\rho(U)<1$ is reached by player Black at some stage of the game, then 
player White has a strategy enforcing the first of player Black's ball $U(l)$ with
\begin{center}
$\rho(U(l))<\rho(U)\mu^{'}$
\end{center}
to satisfy for every $A\in U(l)$
\begin{equation}
|\vec{M}_{N}(A)|>L\sqrt{L}\rho(U)\mu^{'} M_{N -1}(U(l)).
\label{46}
\end{equation}
Alternatively under the same assumptions on $U$, for any $\Y_{1},...,\Y_{M}$  
orthonormal vectors in $\mathbb{R}^{L}$ 
player White has a strategy enforcing the first of player Black's balls $U(l^{'})$ with
\begin{center}
$\rho(U(l^{'}))<\rho(U)\mu^{'}$
\end{center}
to satisfy for every $A\in U(l^{'})$
\begin{equation}
|\vec{M}^{'}_{M}(A)|>L\sqrt{L}\rho(U)\mu^{'} M^{'}_{M -1}(U(l^{'})).
\label{46a}
\end{equation}

\label{corollary of main lemma}

\end{cor}

\begin{proof}
\noindent Replace $\psi$ in lemmas \ref{hard theorem} and \ref{hard theorem 2} 
with $L\sqrt{L}\left(\frac{2}{\epsilon_{0}}\right)^{L}$. 

\noindent Set 
\begin{center}
$\alpha=\min\{\alpha_{1},\alpha_{2}\}$
and $\mu =\min\{\mu_{N},\mu^{'}_{M}\}$.
\end{center}
\noindent Notice that by lemma \ref{hard theorem} (lemma \ref{hard theorem 2}), 
if 
\begin{center}
$0<\mu '\leq \mu _{N}$\hspace{3mm} \text{similarily}\hspace{3mm} $0<\mu '\leq \mu _{M}$ 
\end{center}

\noindent and

\begin{center}
$\rho(U(i_{N}))<\mu '\rho_{0}$ 
\hspace{3mm} \text{similarily}\hspace{3mm}$\rho(U(i_{N}))<\mu '\rho_{0}$,
\end{center}

\noindent then obviously every $A\in U(i_{N})$ ($A\in U(i_{M})$) will satisfy (\ref{46}). 

\end{proof}

\subsubsection{Two geometric lemmas }

\noindent For what follows we shall need to use 
lemmas 1 and 2 in \cite{SSS}.\\

\noindent Let $\X$, $\Y$, ${\cal{A}}(\X)$ and ${\cal{B}}(\Y)$
be as in (\ref{special X}), (\ref{special Y}), (\ref{special A}) and (\ref{special B}).

\noindent Set 
\begin{equation}
\lambda =N/L .
\label{lambda}
\end{equation}

\noindent Given $1<R\in\mathbb{R}$, let

\begin{equation}
\delta =R^{-NL^2},\hspace{5mm} \delta^T=R^{-ML^2},
\label{delta}
\end{equation}

\begin{lemma}
There exists a constant $R_1=R_1(M,N,\sigma )$ such that for every $i\in \mathbb{N}$ and $R\geq R_1$,\\

\noindent if a ball $B$ satisfies 
\begin{center}
$\rho (B)<R^{-L(\lambda+i)}$, 
\end{center} 
with the system

\begin{center}
$0<\xnm<\delta R^{M(\lambda +i)}$
\end{center} 

\begin{center}
$\Anm<\delta R^{-N(\lambda +i)-M}$
\end{center} 
having no solution $\X$ for all $\A_i$ associated with points in $B$,
then the number of linearly independent vectors $\Y$
satisfying the system

\begin{center}
$0<\ynm <\delta^{T} R^{N(1 +i)}$
\end{center} 

\begin{center}
$\Bnm <\delta^{T} R^{-M(1 +i)-N}$
\end{center}

for all $\B_i$ associated with points in $B$ is at most $N$.
\label{Schmidt first lemma}
\end{lemma}

\begin{lemma}
\label{last lemma}
There exists a constant $R_2=R_{2}(M,N,\sigma )$ such that for every $j\in \mathbb{N}$ and $R\geq R_2$,\\

\noindent if a ball $B$ satisfies 

\begin{center}
$\rho (B)<R^{-L(1+j)}$
\end{center}
with the system 

\begin{center}
$0<\ynm <\delta^{T} R^{N(1 +j)}$
\end{center} 

\begin{center}
$\Bnm <\delta^{T} R^{-M(1 +j)-N}$
\end{center}

having no integer solution $\Y$ for all $\B_i$ associated with points in $B$, 
then the number of linearly independent vectors $\X$
satisfying the system

\begin{center}
$0<\xnm <\delta R^{M(\lambda +i)}$
\end{center} 

\begin{center}
$\Anm <\delta R^{-N(\lambda +i)-M}$
\end{center}

for $i=j+1$ for all $\A_i$ associated with points in $B$ is at most $M$.
\label{Schmidt second lemma}
\end{lemma}

\subsubsection{One last lemma}

\noindent We remind that $\rho=\rho(U(0))$ is the first radius chosen by player Black,
and we assume $N\geq M$.

\begin{lemma}
\label{last lemma}
Set $\alpha $ as in lemma \ref{corollary of main lemma}, and given
$0<\beta <1$, let $\mu$ be as in lemma \ref{corollary of main lemma} and $\lambda$ as in (\ref{lambda}).
Then there exists $R=R(M,N,\alpha, \beta, \rho, \tau)$
such that 
player White can direct the game in such a way that for every $i,k\in\mathbb{N}$,  
if $U(k)$ of the game satisfies 
\begin{equation}
\rho (U(k))<R^{-L(\lambda+i)}, 
\label{15}
\end{equation} 
then for all $A\in U(k)$ the system

\begin{equation}
0<\|\x\|_{\infty}<\delta R^{M(\lambda +i)}
\label{16}
\end{equation}

\begin{equation}
\|{\cal{A}}(\X)\|_{\infty}<\delta R^{-N(\lambda +i)-M}
\label{17}
\end{equation}
 
\noindent has no solution $\X$ as in (\ref{special X}), where $\delta$ is as in (\ref{delta}). 

\noindent He can also direct the game such that
for every $i,h\in\mathbb{N}$ 
if $U(h)$ satisfies
 
\begin{equation}
\rho (U(h))<R^{-L(1+i)}, 
\label{18}
\end{equation} 
then for all $A\in U(h)$ the system

\begin{equation}
0<\|\y\|_{\infty}<\delta^{T} R^{N(1 +i)}
\label{19}
\end{equation}

\begin{equation}
\|{\cal{B}}(\Y)\|_{\infty}<\delta^{T} R^{-M(1 +i)-N}.
\label{20}
\end{equation} 

\noindent has no solution $\Y$ as in (\ref{special Y}), where $\delta^{T}$ is as in (\ref{delta}).\\\\
\end{lemma}

\begin{proof}

\noindent In order for lemmas \ref{Schmidt first lemma} and \ref{Schmidt second lemma}
to be applicable, we first demand that 

\begin{center}
$R>max\{R_{1},R_{2}\}$,
\end{center}

\noindent where $R_{1}$ and $R_{2}$ are as defined
in lemmas \ref{Schmidt first lemma} and \ref{Schmidt second lemma}.

\noindent Next, let 
\begin{equation}
\label{nested}
R>max\{\rho^{-\frac{1}{L\lambda}}, (\alpha \beta)^{-\frac{1}{M}}\}.
\end{equation}

\noindent Condition (\ref{nested}) ensures that 
$\rho(U(k_{0}))< \rho$, and the sequence  
\begin{equation}
U(k_{0})\supset U(h_{0})\supset U(k_{1})\supset U(h_{1})\supset\ldots
\label{21}
\end{equation}
is strictly decreasing.

\noindent Finally we demand that   
\begin{equation}
\label{condition on mu}
R>(\alpha\beta\mu)^{-\frac{1}{M}}
\end{equation}

\noindent Let $U(k_{i})$ be the \textbf{first} ball of the game with 
(\ref{15}), and $U(h_{i})$ for the
\textbf{first} ball with (\ref{18}).

\noindent We shall prove the lemma by induction on $i$.

\noindent 1. \underline{base of the induction}.

\noindent We notice that for $i=0$

\begin{center}
$\|\x\|\geq 1>\delta R^{M\lambda}$.
\end{center}

Therefore (\ref{16}) and (\ref{17}) have no solution $\X$ if $A\in U(k_{0})$.

\underline{Induction hypothesis}.
 
\noindent We assume

\begin{center}
$U(0)\supset W(0)\supset\ldots\supset U(k_{0})\supset\ldots \supset U(k_{i})$
\end{center}

have been already chosen such that for every $0\leq j\leq i$, 
(\ref{16}) and (\ref{17}) have no solution for $A\in U(k_{j})$,
and dually we assume that 

\begin{center}
$U(0)\supset W(0)\supset\ldots\supset U(k_{0})\supset\ldots \supset U(k_{i})\supset\ldots\supset U(h_{i})$
\end{center}

have been already chosen such that for every $0\leq j\leq i$,
(\ref{19}) and (\ref{20}) have no solution for $A\in U(h_{j})$.

\vspace{10mm}

Thus it remains to prove that if

\begin{center}
$U(0)\supset W(0)\supset\ldots\supset U(k_{0})\supset\ldots \supset U(k_{i})$
\end{center}

have been already chosen such that for every $0\leq j\leq i$, 
(\ref{16}) and (\ref{17}) have no solution for $A\in U(k_{j})$, player White
can enforce that (\ref{19}) and (\ref{20}) have no solution if $A\in U(h_{i})$.

\noindent Suppose that there are solutions $\Y$ of (\ref{19}) and (\ref{20}) 
with vectors $\B_1,...,\B_N$ associated 
with a point $A$ in $U(k_{i})$.
By our assumptions it is sufficient to consider points $\Y$ satisfying 

\begin{equation}
\delta ^{T}R^{N(1+i-1)}\leq \ynm <\delta ^{T}R^{N(1+i)}.
\label{33}
\end{equation}  

\noindent Thus in particular 
\begin{center}
$\delta ^{T}R^{N(1+i-1)}\leq \|\Y\|_{\infty}$.
\end{center}
By Lemma \ref{Schmidt first lemma}, the vectors $\Y$ will be contained in an $N$-dimensional subspace of $\RL$.
Let $\Y_1,...,\Y_N$ be an orthonormal basis of this subspace and
suppose that the integer point $\Y=t_1\Y_1+...+t_N\Y_N$ satisfies (\ref{20}) and (\ref{33}).

We have that 

\begin{center}
$\delta ^{T}R^{N(1+i-1)}\leq \|\Y\|_{\infty}\leq \left|\Y\right|=\sqrt{t_{1}^{2}+\ldots +t_{N}^{2}}\leq
\sqrt{N}max(\left|t_1 \right|,...,\left|t_N\right|)$.
\end{center}

\noindent And so,

\begin{equation}
\frac{1}{\sqrt{N}}R^{N(1+i-1)}\leq max(\left|t_1 \right|,...,\left|t_N\right|),
\label{35}
\end{equation}

\noindent and

\begin{center}
$\begin{matrix}
\left|t_1(\B_1\cdot\Y_1)+...+t_N(B_1\cdot\Y_N)\right|<\delta ^{T}R^{-M(1+i)-N}\\
.\\
.\\
.\\
\left|t_1(\B_N\cdot\Y_1)+...+t_N(B_N\cdot\Y_N)\right|<\delta ^{T}R^{-M(1+i)-N}.
\end{matrix}$
\end{center}

Let $D$ be the determinant of $(\B_u\cdot\Y_v)_{1\leq u,v\leq N}$, 
and let $D_{uv}$ be the cofactor of $\B_u\cdot\Y_v$
in this determinant.

By Cramer's rule we get for every $1\leq\nu\leq N$

\begin{equation}
\left|t_{v}D\right|\leq N\delta ^{T}R^{-M(1+i)-N}max(\left|D_{1v}\right|,...,\left|D_{Nv}\right|)
\label{cramer}
\end{equation}

and in conjunction with (\ref{35}) we get

\begin{equation}
\left|D\right|\leq N\sqrt{N}R^{-L(1+i)}max(\left|D_{11}\right|,
\left|D_{12}\right|,...,\left|D_{NN}\right|).
\label{last equation}
\end{equation}

\noindent Player White's strategy is to play in such a way 
such that (\ref{last equation}) \underline{is not satisfied} by any $\B_1,...,\B_N$ 
associated with a point $A\in U(h_{i})$. \\

\noindent Set $\rho_{0}=\rho(U(k_{i}))$ and let $0<\mu '$ be chosen  
to satisfy
\begin{equation} 
\mu ^{'}\rho_{0}=R^{-L(1+i)}.
\label{first condition}
\end{equation}

\noindent Notice that by definition,
\begin{center}
$\alpha\beta R^{-L(\lambda +i)}\leq\rho_{0}<R^{-L(\lambda+i)}$.
\end{center}

\noindent and it follows by condition (\ref{condition on mu}) that

\begin{equation}
\mu '<(\alpha\beta)^{-1}R^{L(\lambda+i)-L(1+i)}=(\alpha\beta)^{-1}R^{-M}<\mu.
\label{second condition}
\end{equation}

\noindent Applying lemma \ref{corollary of main lemma},  
player White can enforce the first ball $U(i_{N})=U(h_{i})$ with
\begin{center}
$\rho(U(h_{i}))<\rho_{0}\mu ' =R^{-L(1+i)}$
\end{center}
to satisfy for every $A\in U(i_{N})$ 
\begin{center}
$|\vec{M}_{N}(A)|>L\sqrt{L}\rho_{0}\mu ' M_{N -1}U(i_{N})$.
\end{center}

\noindent Thus for every $A\in U(h_{i})$\\

$|D|=|\vec{M}_{N}(A)|>L\sqrt{L}R^{-L(1+i)}M_{N -1}U(h_{i})>$\\

$> N\sqrt{N}R^{-L(1+i)}max(\left|D_{11}\right|,
\left|D_{12}\right|,...,\left|D_{NN}\right|)$,

\noindent and so (\ref{last equation}) is not satisfied by any $\B_1,...,\B_N$ 
associated with a point $A\in U(h_{i})$.

One can show in almost the same way that if $U({h_{i}})$
has already chosen such that  
(\ref{19}) and (\ref{20}) have no solution for $A\in U(h_{i})$, player White
can enforce $U(k_{i+1})$ to satisfy that for no $A\in U(k_{i+1})$ the system (\ref{16}) and (\ref{17})
has no solution.\\
\end{proof}

\subsubsection{Proof of Theorem \ref{main theorem}}
\begin{proof}

\noindent Let $\alpha$ be as defined in lemma \ref{last lemma}
and let $0<\beta <1$. Once player Black chooses his initial radius $\rho$ for his first ball $U(0)$, 
$R$ as defined in lemma \ref{last lemma} could be chosen by player White.

\noindent Let $\X$ be as defined in (\ref{special X}), i.e.
$\X\in\mathbb{Z}^{L}$ and
 
\begin{center}
$\X=(x_1,...,x_N,...,x_L) : \x=(x_1,...,x_N)\neq (0,...,0)$.
\end{center}

\noindent Then for some $i\in \mathbb{N}$,

\begin{center}
$\delta R^{M(\lambda +i-1)}\leq \|\x\|_{\infty}<\delta R^{M(\lambda +i)}$.
%\label{16}
\end{center}

\noindent By lemma \ref{last lemma},
player White can direct the game in such a way that 
if $U(k_{i})$ of the game satisfies 
\begin{center}
$\rho (U(k_{i}))<R^{-L(\lambda+i)}$, 
\end{center} 
then for all $A\in U(k_{i})$

\begin{center}
$\|{\cal{A}}(\X)\|_{\infty}\geq \delta R^{-N(\lambda +i)-M}$
\end{center}

\noindent Successively applying lemma \ref{last lemma} to ever increasing $i$, player White can direct
the game such that $A=\bigcap_{i=0}^{\infty}U(k_{i})$ will satisfy for every 
$\X$ as defined in (\ref{special X}) 

\begin{center}
$(\|\x\|_{\infty})^{N}(\|{\cal{A}}(\X)\|_{\infty})^{M}\geq \delta^{L}R^{-NM-M^2}$.
\end{center}

\noindent Recalling (\ref{BA def}) we are done, letting 
\begin{center}
$0<C<\delta^{L}R^{-NM-M^2}$.
%\label{sufficient}
\end{center}

\end{proof}

\subsection{Proof of lemma \ref{hard theorem}}

\subsubsection{Preliminaries}
\noindent For the rest of this subsection, we shall need the following notation.\\

\noindent Let $\sigma, \psi , \mu >0$ and suppose that $\nu\in\mathbb{N}$ 
with $0\leq\nu\leq N$. 
Let $U\subset\mathbb{R}^{H}$ be a closed ball
and we denote 
$\rho(U)=\rho_{0}$. 

\noindent We say that $(U, B, \sigma ,M ,N,  \psi , \mu , \nu )$ 
satisfy (*) if 
\begin{enumerate}
\item $\rho_{0}<1$.
\item For every $A\in U$, $|A|\leq \sigma$.
\item $B$ is a closed ball such that $B\subset U$. 
\item $\rho(B)<\mu\rho_{0}$.

\item For any given  
$\Y_{1},...,\Y_{N}$ orthonormal vectors in $\mathbb{R}^{L}$ we have 
for every $A\in B$, 
\begin{center}
$|\vec{M}_{\nu -1}(A)|>\psi\rho_{0}\mu M_{\nu -2}(B)$.
\end{center}
\end{enumerate}

\noindent The next three propositions are proved by Schmidt. See
lemma 5, corollary 1, 
corollary 2 and lemma 6 in \cite{SSS}.\\

\begin{pro}
\label{lemma 5}
Suppose $(U, B, \sigma ,M ,N,  \psi , \mu , \nu )$ 
satisfy (*). There exists $C_{1}=C_{1}(M,N)$ such that for any $\epsilon >0$ 
if $U^{'}$
is a ball contained in $B$ satisfying
\begin{equation}
\rho(U^{'})<\epsilon C_{1}\rho(B),
\label{equation of lemma 5}
\end{equation}
then for any $A^{'}$ and $A^{''}$ in $U^{'}$
\begin{center}
$|\vec{M}_{\nu -1}(A^{'})-\vec{M}_{\nu -1}(A^{''})|<\epsilon\rho_{0}\mu M_{\nu -2}(B)$.
\end{center}

\noindent Furthermore, if
\begin{equation}
\label{50}
\rho({U^{'}})<\frac{1}{2}\psi C_{1}\rho(B)
\end{equation}
then for every $A\in U^{'}$,
\begin{center}
$|\vec{M}_{\nu -1}(A)|\geq \frac{1}{2}M_{\nu -1}(U^{'})$.
\end{center}
\end{pro}

\noindent Before formulating the next two propositions we need the following notation.

\noindent Given $\Y_{1},...,\Y_{N}$ orthonormal vectors 
in $\mathbb{R}^{L}$ and $A\in\RH$ let

\begin{center}
$D_{\nu}(A)=D(\B_{1},...,\B_{\nu})=\det
\left(
\begin{matrix}
\B_{1}\cdot\Y_{1} & . & . & . & \B_{1}\cdot\Y_{\nu}\\
. & . & . & .& . \\
. & . & . & .& . \\
. & . & . & .& . \\
\B_{\nu}\cdot\Y_{1} & . & . & . & \B_{\nu}\cdot\Y_{\nu}
\end{matrix}
\right)$.
\end{center} 

\noindent Thus for every $0\leq \nu\leq N$, 
$D_{\nu}$ is a real polynomial function, $D_{\nu}:\mathbb{R}^{H}\rightarrow \mathbb{R}$
of bounded total degree less than or equal to $N$, and in particular, less than $L$.

\begin{pro}
\label{cor 2}
Suppose $(U, B, \sigma ,M ,N,  \psi , \mu , \nu )$ 
satisfy (*). There exists $C_{2}=C_{2}(M,N)$ such that for any
$\epsilon>0$, if $U^{'}$ satisfies (\ref{equation of lemma 5}), 
then for every $A^{'}$ and $A^{''}$ in $U^{'}$
\begin{center}
$|\nabla D_{\nu}(A^{'})-\nabla D_{\nu}(A^{''})|<C_{2}\epsilon\rho_{0}\mu M_{\nu -2}(B)$.
\end{center}
\end{pro}

\begin{pro}
\label{lemma 6}
Suppose $(U, B, \sigma ,M ,N,  \psi , \mu , \nu )$ 
satisfy (*). There exist $C_{3}=C_{3}(N,\psi)$ and $C_{4}=C_{4}(M,N,\sigma)$ such that 
if $U{'}$ is a ball contained in $B$ satisfying (\ref{50}) 
and $A\in U^{'}$ with 
\begin{equation}
\label{54}
|\vec{M}_{\nu}(A)|<C_{3}\psi M_{\nu -1}(U^{'})
\end{equation}

\noindent and $D_{\nu -1}(A)$
has the largest absolute value among the coordinates of $\M(A)$ then

\begin{equation}
\label{56}
|\nabla D_{\nu}(A)|>C_{4}M_{\nu -1}(U^{'}).
\end{equation}
\end{pro}

\vspace{10mm}

\subsubsection{proof of lemma \ref{hard theorem}}

\begin{proof}

\noindent Given $\psi >0$ define for every $0\leq\nu\leq N$ 

\begin{equation}
\psi_{\nu}=\left(\frac{\epsilon_{0}}{2}\right)^{\nu}\psi. 
\label{psi_N}
\end{equation}

\noindent Assuming $\psi =\psi_{\nu}$ for $0\leq \nu\leq N$ and noticing that in our setup $\sigma =\sigma(\tau)$, let 

\begin{center} 
$C_{3}^{\nu}=C_{3}^{\nu}(N , \psi)$\hspace{5mm}and \hspace{5mm} $C_{4}=C_{4}(M,N,\sigma(\tau))$
\end{center}

\noindent be as in proposition \ref{lemma 6}. 

\noindent Define

\begin{equation}
C_{3}^{Min}=C_{3}^{Min}(N,\psi, \tau)=\min_{1\leq\nu\leq N} C_{3}^{\nu}.
\label{c9min}
\end{equation}

\noindent Given $\psi >0$ let $\psi_{N}$ be as defined in (\ref{psi_N}) 
and let $\alpha_{1}$ be so small as to satisfy

\begin{equation}
(\alpha_{1}) ^{\frac{1}{2}}<\min\{\frac{1}{2},\frac{1}{4} \frac{\psi_{N}\epsilon_{0}}
{N\max_{\{ 0\leq \nu\leq N\} }\binom{N}{\nu}}, C_{3}^{min}, \frac{15}{32}\frac{C_{4}}{\psi} \}.
\label{alpha}
\end{equation}

\vspace{10mm}

\noindent Our initial setup is a closed ball $U\subset\mathbb{R}^{H}$ with
$\rho(U)=\rho_{0}<1$. For every $A\in U$, $|A|\leq \sigma$ for some positive $\sigma =\sigma (\tau)$
and $0<\beta <1$ is given.
We shall prove the lemma by induction on $\nu$. 

\noindent 1. \underline{Base of the induction}. 

\noindent 
For $\nu =0$, $\psi_{0}=\psi$. 
Let $\mu_{0}<\frac{1}{\psi}$ and let $\Y_{1},...,\Y_{N}$ be any set of orthonormal 
vectors in $\mathbb{R}^{L}$ . 
By definition we have 
for \underline{any} ball $V\subset\mathbb{R}^{H}$ and 
any $A\in V$,
\begin{equation}
\vec{M}_{0}(A)=1>\psi_{0}\mu\rho_{0}=\psi_{0}\mu_{0}\rho_{0}M_{-1}(V).
\label{base}
\end{equation}

\noindent 2. \underline{The induction hypothesis}.

\vspace{5mm}
We assume the validity of the lemma for $\nu -1$ ($\nu\geq 1$), i.e., 
there exists $\mu_{\nu -1}$ such that 
player White can play in such a way that
the first of player Black's balls $U(i_{\nu -1})\subset U$
to satisfy 
\begin{center}
$\rho(U(i_{\nu -1}))<\mu_{\nu -1}\rho_{0}$
\end{center}   

\noindent satisfies for every $A\in U(i_{\nu -1})$
\begin{equation}
|\vec{M}_{\nu -1}(A)|>\psi_{\nu -1}\rho_{0}\mu_{\nu -1} M_{\nu -2}(U(i_{\nu -1})).
\label{induction}
\end{equation}

\vspace{5mm}

\noindent We assume that $U(i_{\nu -1})$ with this property is given and thus 
\begin{center}
$(U, U(i_{\nu -1}), \sigma (\tau), M, N,  \psi_{\nu -1} , \mu _{\nu -1}, \nu -1 )$  
\end{center}
satisfy (*) by the induction hypothesis and our initial conditions. 
We shall define $\mu_{\nu}$ and show how player White can play 
in such a way that $U(i_{\nu})$ satisfies (\ref{46}).
\\\\

\noindent Let $j_{\nu}$ 
be the first integer exceeding $i_{\nu -1}$ satisfying
\begin{equation}
\label{3 conditions}
\rho(U(j_{\nu}))<\frac{1}{2}C_{1}\rho (U(i_{\nu -1}))\cdot\min\{\psi_{\nu},  \frac{1}{8}\psi_{N}\frac{C_{4}}{C_{2}} 
\}.
\end{equation}

\noindent (As we shall soon see, $U(j_{\nu})$ will play the part of $U'$ in propositions \ref{lemma 5}-\ref{lemma 6}).

\noindent By definition

\begin{center}
$\rho(U(i_{\nu -1}))\geq\alpha_{1} \beta \rho_{0}\mu_{\nu -1}$,
\end{center}

\noindent and thus there exists
\begin{equation}
c_{\nu -1}=c_{\nu -1}
(M ,N, \alpha_{1}, \beta , \psi ,\tau )
\label{Cnu}
\end{equation}

\noindent such that 

\begin{center}
$\rho(U(j_{\nu}))\geq c_{\nu -1}\rho_{0}$.
\end{center}

\noindent Set 

\begin{equation}
\mu_{\nu}=\mu_{\nu}(M, N, \alpha_{1} , \beta , \psi,\tau)=
\alpha_{1} ^{\frac{1}{2}}c_{\nu -1}
\label{64ii}
\end{equation}

\noindent and

\begin{equation}
K_{\nu -1}=K_{\nu -1}
(M ,N, \alpha_{1}, \beta , \psi ,\tau , \rho_{0})=\frac{\rho(U(j_{\nu}))}{\rho_{0}}.
\label{K}
\end{equation}

\noindent For later use we observe that trivialy $K_{\nu -1}\geq c_{\nu -1}$.

\noindent When $U(i_{\nu -1})$ is given, player White plays in an arbitrary
way until $U(j_{\nu})$ is reached.\\

\noindent \underline{The trivial case }.

\noindent If it so happens that for every $A\in U(j_{\nu})$
\begin{center}
$|\vec{M}_{\nu}(A)|> \psi_{\nu}\rho_{0}\mu_{\nu}M_{\nu -1}(U(j_{\nu}))$,
\end{center}

\noindent player White's strategy is to play in an arbitrary way until the first ball 
$U(i_{\nu})$ to satisfy

\begin{center}
$\rho(U(i_{\nu}))<\mu_{\nu}\rho_{0}$
\end{center}

\noindent is reached, and every $A\in U(i_{\nu})$ will trivially satisfy (\ref{46}).

\noindent \underline{The non-trivial case}.

\noindent Suppose that there exists $A^{'}\in U(j_{\nu})$ such that

\begin{equation}
|\vec{M}_{\nu}(A^{'})|\leq \psi_{\nu}\rho_{0}\mu_{\nu}M_{\nu -1}(U(j_{\nu})).
\label{65}
\end{equation}

\noindent Set
\begin{center}
$\epsilon =\frac{1}{16}\psi_{N}\frac{C_{4}}{C_{2}}(\leq\frac{1}{16}\psi_{\nu -1}\frac{C_{4}}{C_{2}})$.
\end{center}

\noindent We notice that since $U(j_{\nu})\subset U(i_{\nu -1})$ we have by (\ref{induction})

\begin{equation}
M_{\nu -1}(U_{j_{\nu}})>\psi_{\nu -1}\rho_{0}\mu_{\nu -1} M_{\nu -2}(U(i_{\nu -1})).
\end{equation}

\noindent By proposition \ref{cor 2},
for any $A^{'}$ and $A^{''}$ in $U(j_{\nu})$

\begin{equation}
|\nabla D_{\nu}(A^{'})-\nabla D_{\nu}(A^{''})|<C_{2}\epsilon\rho_{0}\mu_{\nu -1} M_{\nu -2}(U(i_{\nu -1}))<
\frac{1}{16}C_{4}M_{\nu -1}(U(j_{\nu})).
\label{63}
\end{equation}

\noindent By (\ref{3 conditions}), $U(j_{\nu})$ satisfies (\ref{50}), and since $\rho_{0}<1$
and $\mu _{\nu}<C_{3}^{\nu -1}$ by (\ref{alpha}) and (\ref{64ii}), 
the point $A^{'}$ satisfies (\ref{54}).

\noindent As no special assumptions were made neither on the $\B_{\nu}$'s nor the $\Y _{\nu}$'s, 
we may assume $D_{\nu -1}(A')$
has the largest absolute value among the coordinates of $\M(A')$. By proposition \ref{lemma 6},

\begin{equation}
|\nabla D_{\nu}(A^{'})|>C_{4}M_{\nu -1}(U(j_{\nu})).
\label{66}
\end{equation}

\noindent Let 

\begin{equation}
D^{'}=\nabla D_{\nu}(A^{'}).
\label{67}
\end{equation}

\noindent Denote the center of $U(j_{\nu})$ by $A(j_{\nu})$. \\

\noindent If 

\begin{equation}
D_{\nu}(A(j_{\nu}))\geq 0,
\label{68}
\end{equation}

\noindent let

\begin{center}
$A_{M}=A(j_{\nu})+(1-\alpha_{1})\frac{1}{|D^{'}|}\rho(U(j_{\nu}))D^{'}$.
\end{center}

\noindent Thus

\begin{center}
$\left(A_{M}-A(j_{\nu})\right)\cdot D^{'}= (1-\alpha_{1})\rho(U(j_{\nu}))|D^{'}|$.
\end{center}

\noindent Since $\alpha_{1} <\frac{1}{4}$ we have 

\begin{equation}
\left(A_{M}-A(j_{\nu})\right)\cdot D^{'}>\frac{3}{4}\rho(U(j_{\nu}))|D^{'}|.
\label{70}
\end{equation}

\noindent In view of (\ref{68}), (\ref{63}), (\ref{70}) and (\ref{64ii})

\vspace{10mm}

$D_{\nu}(A_{M})\geq D_{\nu}(A_{M})-D_{\nu}(A(j_{\nu}))=$

$\int\limits_0^1(A_{M}-A_{j_{\nu}})\cdot (\nabla D_{\nu}((1-s)A(j_{\nu})+sA_{M}))ds=$

$(A_{M}-A(j_{\nu}))\cdot D^{'}+\int\limits_0^1((A_{M}-A(j_{\nu}))\cdot 
(\nabla D_{\nu}((1-s)A(j_{\nu})+sA_{M})-D^{'}))ds$

$\geq \frac{3}{4}\rho(U(j_{\nu}))|D^{'}|-2(1-\alpha_{1})\rho(U(j_{\nu}))\frac{1}{16}C_{4}M_{\nu -1}(U(j_{\nu}))$\\

$>\frac{15}{32}C_{4}K^{\nu -1}\rho_{0}M_{\nu -1}(U(j_{\nu}))>
\alpha_{1}^{\frac{1}{2}} K^{\nu -1}\psi_{\nu -1}\rho_{0}M_{\nu -1}(U(j_{\nu}))$.\\

\noindent Thus
\begin{center}
$D_{\nu}(A_{M})>\alpha_{1}^{\frac{1}{2}} K^{\nu -1}\psi_{\nu -1}\rho_{0}M_{\nu -1}(U(j_{\nu}))$.
\end{center}

\noindent In the case 
\begin{center}
$D_{\nu}(A(j_{\nu}))< 0$,
\end{center}

\noindent we let 
\begin{center}
$A_{M}=A(j_{\nu})-(1-\alpha_{1})\frac{D^{'}}{|D^{'}|}\rho(U(j_{\nu}))$,
\end{center}

\noindent and we get
\begin{center}
$-D_{\nu}(A_{M})<
-\alpha_{1}^{\frac{1}{2}} K^{\nu -1}\psi_{\nu -1}\rho_{0}M_{\nu -1}(U(j_{\nu}))$.
\end{center}

\noindent Combining we get

\begin{equation}
|D_{\nu}(A_{M})|>\alpha_{1}^{\frac{1}{2}} K^{\nu -1}\psi_{\nu -1}\rho_{0}M_{\nu -1}(U(j_{\nu})).
\label{norm on W}
\end{equation}

\noindent Let
\begin{center}
$\Omega=B(A(j_{\nu}),(1-\alpha_{1})\rho(U(j_{\nu})))$.
\end{center}

\noindent Since $A_{M}\in \Omega$, we conclude by (\ref{norm on W})

\begin{equation}
\|D_{\nu}\|_{\Omega}>\alpha_{1}^{\frac{1}{2}} K^{\nu -1}\psi_{\nu -1}\rho_{0}M_{\nu -1}(U(j_{\nu})).
\label{norm}
\end{equation}

\vspace{5mm}

\noindent By (\ref{good}) and (\ref{epsilon_0}) we have

\begin{center}
$\tau\left(\{A\in \Omega :|D_{\nu}(A)|
<\epsilon_{0}\alpha_{1}^{\frac{1}{2}} K^{\nu -1}\psi_{\nu -1}\rho_{0}M_{\nu -1}(U(j_{\nu}))\}\right)
\leq \frac{1}{2}\tau (\Omega)$.
\end{center}

\noindent Thus there exists $A_{0}\in \Omega\cap supp(\tau)$ such that 

\begin{equation}
|\vec{M}_{\nu}(A_{0})|\geq|D_{\nu}(A_{0})|\geq\epsilon_{0}\alpha_{1}^{\frac{1}{2}} 
K^{\nu -1}\psi_{\nu -1}\rho_{0}M_{\nu -1}(U(j_{\nu})),
\label{first}
\end{equation}

\noindent and player White chooses a ball $W(j_{\nu})=B(A_{0},\alpha_{1}\rho(U_{j_{\nu}}))$.
Assume $A\in W(j_{\nu})$.

\noindent Notice that every coordinate of $\vec{M}_{\nu}(A)$ is a certain determinant of a
$\nu \times \nu $ matrix
depending on some $\gamma_{ij}$. The absolute values of the 
partial derivatives of every such determinant are no greater then 

\begin{center}
$N|\vec{M}_{\nu -1}(A)|\leq NM_{\nu -1}(U(j_{\nu}))$.
\end{center}

\noindent Set $C=N\cdot{\max_{\{ 0\leq \nu\leq N\} }\binom{N}{\nu}}$. 
By elementary calculus, (mean value theorem), 

\begin{center}
$|\vec{M}_{\nu}(A_{0})-\vec{M}_{\nu}(A)|
\leq\sqrt{({\max_{\{ 0\leq \nu\leq N\} }\binom{N}{\nu}})^{2}}\cdot 2\alpha_{1}\rho(U(j_{\nu}))\cdot
NM_{\nu -1}(U(j_{\nu}))
=2C\alpha_{1}\rho(U(j_{\nu}))M_{\nu -1}(U(j_{\nu}))
<2C\frac{\psi_{N}\epsilon_{0}}{4C}\alpha_{1}^{\frac{1}{2}}K^{\nu -1}\rho_{0}M_{\nu -1}(U(j_{\nu}))
\leq\frac{1}{2}\epsilon_{0}\psi_{\nu -1}\alpha_{1}^{\frac{1}{2}}K^{\nu -1}\rho_{0}M_{\nu -1}(U(j_{\nu}))$.
\end{center}

\noindent Combining with (\ref{first}) we get for every $A\in W(j_{\nu})$,
\begin{center}
$|\vec{M}_{\nu}(A)|>\frac{1}{2}\epsilon_{0}\psi_{\nu -1}\alpha_{1}^{\frac{1}{2}}K^{\nu -1}\rho_{0}M_{\nu -1}(U(j_{\nu}))
\geq\psi_{\nu}\mu_{\nu}\rho_{0}M_{\nu -1}(U(j_{\nu}))$.
\end{center}

\noindent We conclude that every $A\in U(j_{\nu}+1)$ satisfies (\ref{46}) and 
the first ball $U(i_{\nu})$ satisfying 

\begin{center}
$\rho(U(i_{\nu}))<\mu_{\nu}\rho_{0}$,
\end{center}

\noindent will satisfy (\ref{46}). 
Player White can play in an arbitrary way until 
such a ball is reached by player Black.

\end{proof}

\section{Application to fractals}

\noindent A map $\phi : \RN \rightarrow \RN$ is a  \textbf{similarity}
if it can be written as
\begin{center}
$\phi (\x) = \rho \Theta (\x) + \y,$
\end{center}
where $\rho \in \mathbb R^{+}$, $\Theta \in O (N,\mathbb R)$ 
and $\y \in \RN $.
It is said to be \textbf {contracting} if $\rho < 1$.
It is known (see \cite{H} for a more general statement) that for
any finite family
$\phi_1,\dots, \phi_m$ of  contracting similarities there exists a
unique nonempty compact set ${\cal{K}}$,
called the \textbf{attractor} or \textbf{limit set} of the family,
such that
\begin{center}
${\cal{K}} = \bigcup_{ i = 1 }^m\phi_i ( {\cal{K}} ).$
\end{center}
Say that
$\phi_1, \dots,  \phi_m$ as above satisfy the {\textbf{open set
condition} (first introduced in \cite{H}) 
if there exists an open subset $U \subset \RN $ such that
\[\phi_i ( U ) \subset U \ \mathrm{for \ all \ \ } i=1, \ldots,
m\,,\]
          and
\[i \ne j \Longrightarrow \phi_i ( U ) \cap \phi_j(U) =
\varnothing\,.\]
The family $\{\phi_i\}$ is called \textbf{irreducible} if there is no
finite collection of proper affine subspaces which is invariant under
each $\phi_i$.
Well-known self-similar sets, like Cantor's ternary set, Koch's curve
or Sierpinski's gasket, are all examples of attractors of irreducible
families of
contracting similarities satisfying the open set condition.\\

\noindent As an immediate consequence of theorem \ref{main theorem}
we have,

\begin{cor}
\label{application}
Let $\left\{ \phi_1,...,\phi_k\right\}$ 
be a finite irreducible family of 
contracting similarity maps of $\mathbb{R}^{MN} $ satisfying the open set condition
with ${\cal{K}}$ its attractor. Then 
\begin{center}
dim($\textbf{BA}(M,N)\cap {\cal{K}}$)=dim ${\cal{K}}$.
\end{center}
\end{cor}

\begin{proof}

Let $\delta $ be
the Hausdorff dimension of
${\cal{K}}$, and $\tau$ the restriction of the $\delta $-dimensional Hausdorff
measure to ${\cal{K}}$. 

It is known that $\tau $ is an absolutely friendly measure.
(See \cite{KLW}(Theorem 2.3, Lemma 8.2 and 8.3)).
Furthermore, it was proved in \cite{LF} (Corollary 5.3) that for this particular case a winning set enjoys
full dimension. 

\end{proof}

\section{Windim}

Let $M$ be a complete metric space. Define the winning dimension of $S\subset M$, Windim($S$), as
follows. If $S$ is $\alpha$-winning for no $\alpha >0$ then Windim($S$)=0. 

\noindent Otherwise Windim($S$) is
the least upper bound on all $0<\alpha<1$ such that $S$ is $\alpha$-winning. 

\noindent In \cite{SSS} Schmidt was able to prove that the winning dimension of $\textbf{BA}(M,N)$
in $\mathbb{R}^{H}$ is $\frac{1}{2}$. 
This is the best possible result for any proper subset of $\mathbb{R}^{H}$.
In this paper as well in \cite{LF}, no upper bound on 
the winning dimension of $\textbf{BA}(M,N)\cap {\cal{K}}$
(where ${\cal{K}}$ is as defined in corollary \ref{application}) is given and
a natural question would be whether one could improve the proof leading to the 
optimal upper bound of $\frac{1}{2}$ similar to the case in \cite{SSS}.
In what follows we prove that in general one cannot.\\

\noindent Let $C$ denote the usual middle third Cantor set and we remind that $\textbf{BA}(1,1)$
is the set of badly approximable numbers. 

\begin{pro}
Windim$(C\cap \textbf{BA}(1,1))\leq \frac{1}{3}$
\end{pro}

\begin{proof}
\noindent For every $k\in\mathbb{N}$ we denote by $U(k)$ (respectively $W(k)$)
player Black's kth ball choice ( respectively player White's kth ball choice).

\noindent Given any $\alpha >\frac{1}{3}$, we prove that player Black can 
always pick $\beta =\beta(\alpha)$ and specify a strategy such that 
$\cap_{k=0}^{\infty}W(k)=\cap_{k=0}^{\infty}W(k)$ is not a badly approximable number.

\noindent Given $\alpha =\frac{1}{3} +\epsilon$ for some $\epsilon >0$, 
there exists $N\in\mathbb{N}^+$ such that $\frac{1}{3^{N}}<\epsilon $. 

\noindent Thus it is sufficient to prove the proposition for 
$\alpha=\frac{1}{3}+\frac{1}{3^N}$ 
for any given $N$.

\noindent Given $\alpha =\frac{1}{3}+\frac{1}{3^{N}}$ let

\begin{center}
$\beta =\frac{1}{3^{N-1}+1}$ \hspace{5mm} and \hspace{5mm} $U(0)=B(0,1)$.
\end{center}

\noindent We have, 

\begin{center}
$\rho(U(k))=\frac{1}{3^{Nk}}$ \hspace{5mm} and \hspace{5mm} $\rho(W(k))=\frac{1}{3^{Nk+1}}+\frac{1}{3^{N(k+1)}}$.
\end{center}
We prove by induction on $k$ that player Black can 
always choose $0$ as the center of his balls. 
Obviously for $k=0$ the condition is satisfied.

\noindent Assume for $k$, i.e., $U(k)=B(0, \frac{1}{3^{Nk}})$. 
As $\rho(W(k))=\frac{1}{3^{Nk+1}}+\frac{1}{3^{N(k+1)}}$, the rightmost point player White can choose for his center
is $\frac{1}{3^{Nk+1}}$. To see this notice that the next point bigger then $\frac{1}{3^{Nk+1}}$ in $C$
is $\frac{2}{3^{Nk+1}}$. But player White can't choose this point on account that his radius is too large, i.e., 
$\rho (W(k))>\frac{1}{3^{Nk+1}}$.

\noindent We notice that the closed interval 
$[-\frac{1}{3^{N(k+1)}},\frac{1}{3^{N(k+1)}}+\frac{1}{3^{Nk+1}}]\subset W(k)$
for any choice of player White's ball $W(k)$. In particular,

\begin{center}
$U(k+1)=B(0,\frac{1}{3^{N(k+1)}})\subset W(k)$ for any possible $W(k)$.
\end{center}

\end{proof}

\section{Measure of Intersection}

%In this section, let \textbf{BF}(M,N) denote the set of either badly approximable numbers, vectors

The result regarding the intersection's dimension of $\textbf{BA}(M,N)$ and the compact support
of an absolutely friendly measure on $\RH$ (where $H=M\cdot N$) is similar to that of the dimension of 
$\textbf{BA}(M,N)$, i.e., in both cases the result is a full dimension.
It is well known that $\lambda(\textbf{BA}(M,N))=0$ where $\lambda$ is the Lebesgue measure in $\RH$.
It is therfore logical to ask the following:   
\begin{question}
Let $\tau$ be an absolutely friendly measure on $\RH$ supported on ${\cal{K}}$, a compact subset of $\RH$.
Is it true in general that
\begin{center}
$\tau({\cal{K}}\cap\textbf{BA}(M,N))=0$?
\end{center}
\end{question}

\noindent The answer is negative as the following example demonstrates.\\

\noindent Consider continued fraction expensions 
%\begin{center}
$\left[0;a_{1},a_{2},...a_{n},...\right]$
%\end{center}
such that for every $i\in\mathbb{N}$, $a_{i}\in \{1,3\}$.

\noindent Let $A_{a_{1},a_{2},...a_{n}}$ denote the interval containing all numbers with continued fraction
expention initial segment $\left[0;a_{1},a_{2},...a_{n}\right]$.

\noindent Thus for example $A_{1}=[\frac{1}{2},1]$, $A_{3}=[\frac{1}{4},\frac{1}{3}]$ and
$A_{1,1,1,3}=[\frac{7}{11}, \frac{9}{14}]$.

\noindent Let $I=[0,1]$ be the 0'th stage of the construction, $A_{1}$ and $A_{3}$ belong to the first stage
of the construction $E_{1}$, $A_{1,1}$, $A_{1,3}$ $A_{3,1}$ $A_{3,3}$ to the second, $E_{2}$ and so on.

\noindent By definition, if there exists $i\in\mathbb{N}$, $1\leq i\leq n$ such that $a_{i}\neq a'_{i}$, then 

\noindent $A_{a_{1},a_{2},...a_{i},...,a_{n}}\cap A_{a_{1},a_{2},...a'_{i},...,a_{n}}=\emptyset$.

\noindent We recall that if we recursively define
\begin{equation}
q_{-1}=0,\hspace{10mm} q_{0}=1 \hspace{10mm}\text{and} \hspace{10mm}q_{n}=a_{n}q_{n-1}+q_{n-2},
\label{q}
\end{equation}   

\noindent then 
\begin{center}
$l(A_{a_{1},a_{2},...a_{n}})=\frac{1}{q_{n}(q_{n}+q_{n-1})}$, 
\end{center}

\noindent where $l(A_{a_{1},a_{2},...a_{n}})$ is the length of the interval $A_{a_{1},a_{2},...a_{n}}$.

\begin{claim}
For any $n\in\mathbb{N}$,
\begin{center}
$\frac{1}{12}<\frac{l(A_{a_{1},a_{2},...a_{n},a_{n+1}})}{l(A_{a_{1},a_{2},...a_{n}})}<\frac{1}{2}$.
\end{center}
\end{claim}

\begin{proof}
\noindent By (\ref{q}), $l(A_{a_{1},a_{2},...a_{n},3})<l(A_{a_{1},a_{2},...a_{n},1})$.   
Thus, 
\begin{center}
$\frac{l(A_{a_{1},a_{2},...a_{n},a_{n+1}})}{l(A_{a_{1},a_{2},...a_{n}})}
=\frac{q_{n}(q_{n}+q_{n-1})}{q_{n+1}(q_{n+1}+q_{n})}
\leq \frac{q_{n}(q_{n}+q_{n-1})}{(q_{n}+q_{n-1})(2q_{n}+q_{n-1})}
=\frac{q_{n}}{2q_{n}+q_{n-1}}<\frac{1}{2}$.
\end{center}

\noindent On the other hand, 
\begin{center}
$\frac{l(A_{a_{1},a_{2},...a_{n},a_{n+1}})}{l(A_{a_{1},a_{2},...a_{n}})}
=\frac{q_{n}(q_{n}+q_{n-1})}{q_{n+1}(q_{n+1}+q_{n})}
\geq \frac{q_{n}(q_{n}+q_{n-1})}{(3q_{n}+q_{n-1})(4q_{n}+q_{n-1})}>\frac{1}{12}$,
\end{center}

\noindent by a simple calculation and using (\ref{q}).\\

\noindent We define the measure $\tau$ as follows: 
\begin{center}
$\tau(A_{a_{1},a_{2},...a_{i},...,a_{n}})=\frac{1}{2^{n}}$. 
\end{center}

\noindent Follwing W. A. Veech, \cite{V} (section 2, proposition 2.5) and \cite{KW} (section 6, remark 6.2 - 
in which a generalization of Veech's definitions and results are discussed in relation to the friendly conditions), 
it is easily checked that $\tau$ is an absolutely friendly measure and obviously,

\begin{center}
$\tau(\textbf{BA}(1,1)\cap supp(\tau))=1$.
\end{center}

\end{proof}

\end{document}